\documentclass[12pt]{amsart}

\pdfoutput=1

\usepackage[margin = 1in]{geometry}
\usepackage[english]{babel}
\usepackage[latin1]{inputenc}
\usepackage[T1]{fontenc}
\usepackage{amssymb,amsmath,amstext,amsfonts,amsthm,braket}
\usepackage{mathtools}
\usepackage{verbatim}
\usepackage{relsize}
\usepackage[colorlinks=true, urlcolor=blue, linkcolor=blue, citecolor=blue]{hyperref}
\usepackage{fancyvrb}
\usepackage{tikz-cd}
\usepackage[percent]{overpic}
\usepackage{comment}
\usepackage{algpseudocode}
\usepackage{float}
\usepackage{xcolor}
\usepackage{cleveref}
\usepackage{tikz}
\usepackage[justification=centering]{caption}

\usepackage{todonotes}

\usepackage[normalem]{ulem}

\usepackage{enumitem}

\numberwithin{equation}{section}
\theoremstyle{plain}
\newtheorem*{algorithm}{Pseudo--Code}
\newtheorem*{theorem*}{Theorem}
\newtheorem{theorem}{Theorem}
\numberwithin{theorem}{section}
\newtheorem{proposition}[theorem]{Proposition}
\newtheorem{lemma}[theorem]{Lemma}
\newtheorem{corollary}[theorem]{Corollary}

\theoremstyle{definition}
\newtheorem{definition}[theorem]{Definition}

\newtheorem{example}[theorem]{Example}

\DeclareMathOperator{\im}{im}
\DeclareMathOperator{\interior}{int}

\DeclareMathOperator{\rank}{rank}
\DeclareMathOperator{\supp}{supp}



\title{Completions to discrete probability distributions in log-linear models}

\author[M.~Cai]{May Cai}
\address{M.~Cai\\
            School of Mathematics\\
            Georgia Institute of Technology\\
            Atlanta\\
            Georgia \ 30332\\
            USA}
\email{mcai@gatech.edu}

\author[C.O. ~Recke]{Cecilie Olesen Recke}
\address{C.O.~Recke\\
         Department of Mathematics\\
         University of Copenhagen\\
         Copenhagen \\ 
         Denmark}
\email{cor@math.ku.dk}

\author[T.~Yahl]{Thomas Yahl}
\address{T.~Yahl\\
        Department of Mathematics\\
         University of Wisconsin-Madison\\
         Madison\\
         Wisconsin \ 53706\\
         USA}
\email{tyahl@wisc.edu}
\urladdr{https://tjyahl.github.io/}

\date{}

\begin{document}

\maketitle

\begin{abstract}
Completion problems, of recovering a point from a set of observed coordinates, are abundant in applications to image reconstruction, phylogenetics, and data science. We consider a completion problem coming from algebraic statistics: to describe the completions of a point to a probability distribution lying in a given log-linear model. When there are finitely many completions, we show that these points either have a unique completion or two completions to the log-linear model depending on the set of observed coordinates. We additionally describe the region of points which have a completion to the log-linear model.

\vspace{\baselineskip}

\noindent\textbf{Keywords:} Log-linear model, toric variety, completion, algebraic moment map, semialgebraic set, algebraic boundary

\noindent\textbf{MSC2020 subject classification:} Primary: 62R01, Secondary: 14M25, 14P10, %14N10, 62D10
\end{abstract}

%%%%%%%%%%%%%%%%%%%%%%%%%%%%%%%%%%%%%%%%%%%%%%%%%%%%%%
%%%%INTRODUCTION
\section{Introduction}
We consider the problem of recovering a probability distribution from partial information. This may occur as an imperfect sampling method may prevent one from observing or distinguishing certain outcomes and thus, it may be that probabilities are only known for certain outcomes. With a priori knowledge that the probability distribution belongs to a specified statistical model, the known probabilities may be used to compute the probability of each outcome. In this case, we say the original probability distribution may be recovered, or completed.

A problem of this form is known as a \textit{probability completion problem}. We describe the general setting more precisely. We consider discrete probability distributions with outcome states $[n] = \{1,\dotsc,n\}$ for a fixed positive integer $n$. Such a probability distribution may be represented by a tuple $(p_1,\dotsc,p_n)\in\mathbb{R}^n$, whose $i$-th coordinate $p_i$ is the probability of outcome $i\in[n]$. The probability simplex $\Delta_{n-1}\subseteq\mathbb{R}^n$ is the set of these discrete probability distributions and a \textit{statistical model} $\mathcal{M}\subseteq\Delta_{n-1}$ is a subset of the probability simplex. A fixed subset of the states $E\subseteq[n]$ will index the probabilities which are to be known or observed, and we consider the coordinate projection $\pi_E:\mathcal{M}\to\mathbb{R}^E$ to the coordinates indexed by $E$. A point $p_E\in\mathbb{R}^E$ is called a \textit{partial observation} and a probability distribution in the fiber $p\in\pi_E^{-1}(p_E)$ is a \textit{completion} of the partial observation $p_E$ to the model $\mathcal{M}$. One looks to describe the fiber $\pi_E^{-1}(p_E)$, which is the set of completions of the partial observation $p_E$ to $\mathcal{M}$. In addition to enumerating the completions of a partial observation, one may look to explicitly describe the \textit{completable region} $\pi_E(\mathcal{M})$---the set of partial observations which may be completed to $\mathcal{M}$.

Probability completion problems were first considered in \cite{kahle2017geometry,kubjas2014matrix}, where they study completions to the independence model of two or more random variables. As the independence model is the intersection of the space of rank one tensors and the probability simplex, this may be regarded as a type of low rank tensor completion problem as in \cite{singer2010uniqueness,kiraly2012algebraic,kiraly2015algebraic,bernstein2020typical}. More generally, probability completion problems may be understood as problems in compressed sensing as described in \cite{breiding2023algebraic}.

We demonstrate an example of a probability completion problem: the Hardy-Weinberg curve used in genetics is a statistical model whose probability distributions represent the probabilities of passing on certain traits from parents to their offspring. In the case of a trait with a dominant gene $X$ and a recessive gene $Y$, there are three genotypes that can be passed on---these are the homozygous combinations $XX$ and $YY$, and the heterozygous combination $XY$. Using the variables $x$ and $y$ for the probabilities of passing on the homozygous combinations $XX$ and $YY$ respectively, and ussingular locus of maping $z$ for the probability of passing on the heterozygous combination $XY$, the Hardy-Weinberg model $\mathcal{M}$ of possible probability distributions is defined by the equations $z^2-4xy = 0$ and $x+y+z=1$, where all coordinates are non-negative. This curve is depicted in Figure \ref{fig:HardyWeinberg}.

\begin{figure}[H]
    \centering
    \includegraphics[scale=.4]{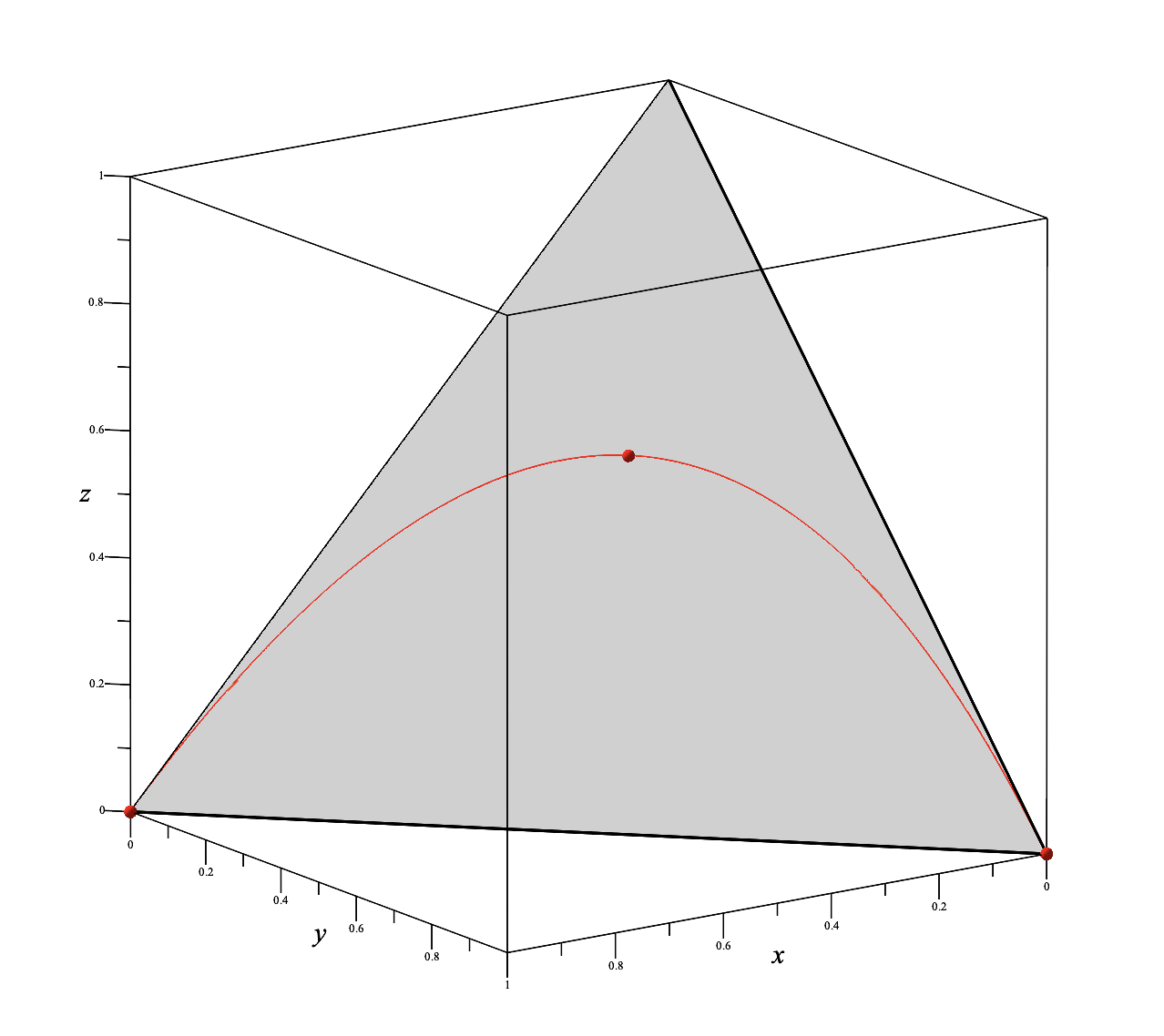}
    \caption{The Hardy-Weinberg curve}
    \label{fig:HardyWeinberg}
\end{figure}

The projections onto either the first coordinate or the second coordinate are injective and the image in both cases is the interval $[0,1]$. That is, a point on the Hardy-Weinberg curve is uniquely completable from its first coordinate or its second coordinate. Equivalently, if the probability of either homozygous combination $XX$ or $YY$ being passed on is known, all probabilities can be determined. However, the projection onto the third coordinate is a $2$-to-$1$ mapping for $0\le z<1/2$ and the image of the projection is $[0,1/2]$. Thus, from a known $z$ coordinate less than $1/2$, there are two completions to a point of the Hardy-Weinberg curve. Equivalently, given that the probability of the heterozygous combination $XY$ being passed on is known and less than $1/2$ there are two possible probabilities for the homozygous combinations $XX$ and $YY$ to be passed on.

We are concerned with probability completion problems where the statistical model $\mathcal{M}$ is a log-linear model, the restriction of a toric variety to the probability simplex. The class of log-linear models encompasses many well-studied discrete models such as discrete graphical models, hierarchical models, and staged-tree models \cite{serkan2002grobner,serkan2007finiteness,aida2022staged}. Log-linear models are also useful as their Markov bases are understood and may be used in sampling algorithms as described in \cite{sturmfels1998algorithms,sullivant2018algebraic}.

Our main contributions are in showing that the behaviour demonstrated for the Hardy-Weinberg curve is typical. Precisely, we prove that for a log-linear model $\mathcal{M}$ and suitable set of observed outcomes $E\subseteq[n]$, there is either a unique completion or there are two completions for every point with non-zero coordinates in the completable region. Further, we use the relationship between toric varieties and polyhedral geometry to identify precisely when there is one completion and when there are two completions. These results are given in Theorem \ref{thm:boundary-completions} and Theorem \ref{thm:interior-completions}.

In addition, we give a description of the boundary and interior of the completable region in Theorem \ref{thm:boundary-interior-image}. We use this result to provide an algorithm for computing the defining equations for the boundary of the completable region. These defining equations often allow one to compute an explicit semialgebraic description of the completable region, as discussed in Section 5. We illustrate these results with several examples, some of which are relevant for applications.

We begin by giving background on toric varieties and log-linear models in Section~\ref{sec:log-linear models}. In Section~\ref{sec:toricresults}, we describe the completion problem to a given toric variety and include necessary results for future sections. We present our completion results to a given log-linear model in Section~\ref{sec:loglinearresults}. Last, in Section~\ref{sec:computing-completable-region}, we give a procedure for describing the algebraic boundary of the completable region, as well as provide explicit semialgebraic descriptions for certain models.

\subsection*{Acknowledgements} The authors would like to thank Kaie Kubjas for her mentorship and advice throughout the research process. We are also grateful for Serkan Ho{\c s}ten, Kaie, and Bernd Sturmfels for organizing the Varieties from Statistics Apprenticeship Week at the Institute for Mathematical and Statistical Innovation (IMSI). Furthermore we would like to thank all the other participants of the Algebraic Statistics and Our Changing World program for their feedback on early drafts of the paper. Part of this research was performed while the authors were visiting the Institute for Mathematical and Statistical Innovation (IMSI), which is supported by the National Science Foundation (Grant No.\ DMS-1929348). This research was partially supported by NSF Grant No.\ DMS-1855726. Cecilie Olesen Recke was supported by Novo Nordisk Foundation Grant NNF20OC0062897.

%%%%%%%%%%%%%%%%%%%%%%%%%%%%%%%%%%%%%%%%%%%%%%%%%%%%%%
%%%%TORIC VARIETIES
\section{Toric Varieties}\label{sec:log-linear models}
We discuss the necessary background on toric varieties for the completion problem. Given a vector of non-negative integers $v = (v_1,\dotsc,v_k)\in\mathbb{Z}_{\ge 0}^k$ and variables $\theta = (\theta_1,\dotsc,\theta_k)$, a \textit{monomial} is an expression of the form $\theta^v = \theta_1^{v_1}\dotsb\theta_k^{v_k}$. The vector $v$ is the exponent vector of the monomial $\theta^v$. 

Let $A\in\mathbb{Z}_{\ge 0}^{k\times n}$ be an integer matrix with columns $a_1,\dotsc,a_n\in\mathbb{Z}_{\ge 0}^k$. We make the assumption that the \textit{column sums of $A$ are equal} to a positive integer $N>0$. This is the case for many meaningful statistical models in applications, such as discrete graphical models, hierarchical models, and staged-tree models \cite{serkan2002grobner,serkan2007finiteness,aida2022staged}. We define a map $\varphi^A:\mathbb{C}^k\to\mathbb{C}^n$ by
\begin{align*}
\varphi^A(\theta) = (\theta^{a_1},\dotsc,\theta^{a_n}).
\end{align*}
The coordinate functions of $\varphi^A$ are monomials, and the exponent vector of each monomial is a column of $A$. Thus, the map $\varphi^A$ is homogeneous in that $\varphi^A(\lambda p) = \lambda^N\varphi^A(p)$ and its image is a cone. 

The \textit{toric variety} $X_A$ associated to the integer matrix $A\in\mathbb{Z}_{\ge 0}^{k\times n}$ is the Zariski closure of the image $X_A = \overline{\im\varphi^A}$. The toric variety $X_A$ is an irreducible variety of dimension $\dim X_A = \rank A$. We remark that in the more general language of \cite{CoxLittleSchenck}, the set $\varphi^A((\mathbb{C}^\times)^k)$ is a dense torus which acts on the toric variety $X_A$ by coordinate-wise multiplication. Precisely, if $p\in X_A$ lies in the toric variety and $\theta\in(\mathbb{C}^\times)^k$, then the coordinate-wise product $p\varphi^A(\theta)\in X_A$ also lies in the toric variety. This is the origin of the term ``toric variety''---the variety $X_A$ contains a dense open set which is isomorphic to an algebraic torus and whose action on itself extends to the variety $X_A$.

The \textit{toric ideal} $I_A$ is the defining ideal of a toric variety $X_A$,
\begin{align*}
I_A = I(X_A) = \{f\in\mathbb{C}[x_1,\dotsc,x_n]:f(x) = 0~\text{ for all }~x\in X_A\}.
\end{align*}
The toric ideal $I_A$ is a prime ideal generated by pure binomials and generators for this ideal may be computed explicitly from the matrix $A$, as described in \cite[Proposition 6.2.4]{sullivant2018algebraic} restated here for convenience.

\begin{proposition}\label{ToricIdeal}
If $A\in\mathbb{Z}_{\ge 0}^{k\times n}$ is an integer matrix, then
\begin{align*}
    I_A = \langle p^{u} - p^{v} \; | \; u,v \in \mathbb{N}^n \text{ and } Au = Av  \rangle.
\end{align*}
Further, since the column sums of $A$ are equal, the ideal $I_A$ is a homogeneous ideal.
\end{proposition}

There are software such as \texttt{4ti2} \cite{4ti2} and \texttt{Macaulay2} \cite{M2} that contain methods used to effectively compute Gr\"{o}bner bases for toric ideals. This allows one to make several computations with toric ideals, such as determining containment of points in a toric variety and computing elimination ideals.

%\subsection{Non-negative Toric Varieties and Log-Linear Models} 
We write $\mathbb{R}_{\ge 0}^n$ and $\mathbb{R}_{>0}^n$ for the set of points in $\mathbb{R}^n$ with non-negative coordinates and positive coordinates respectively. The \textit{probability simplex} 
\begin{align*}
\Delta_{n-1} = \{(p_1,\dotsc,p_n)\in\mathbb{R}_{\ge 0}^n: \sum_{i=1}^n p_i = 1\}
\end{align*}
is the set of probability distributions in $\mathbb{R}^n$. Our statistical models of interest are intersections of toric varieties with the probability simplex.

\begin{definition}
    The \textit{log-linear model} defined by an integer matrix $A\in\mathbb{Z}_{\ge 0}^{k\times n}$ is the set
    \begin{align*}
        \mathcal{M}_A = X_A \cap \Delta_{n-1}.
    \end{align*}
    The set $\mathcal{M}_A^{>0}$ is the set of points in the log-linear model $\mathcal{M}_A$ with non-zero coordinates.
\end{definition}

Our definition of log-linear model differs from that found in \cite[Chapter 6]{sullivant2018algebraic}. Indeed, we allow for probability distributions having zero coordinates, lying in the boundary of the probability simplex. The name ``log-linear'' originates as for those points $p\in\mathcal{M}_A^{>0}$ with non-zero coordinates, the coordinate-wise logarithm $\log(p)$ lies in the linear space $\im A^T$. Many familiar discrete probability models are in fact log-linear models, such as the independence model, undirected graphical models, and hierarchical models. We describe the topology of the statistical model $\mathcal{M}_A$ via the real structure of the toric variety $X_A$.

Given a matrix $A\in\mathbb{Z}_{\ge 0}^{k\times n}$, the \textit{non-negative toric variety} $X_A^{\ge 0} = X_A \cap \mathbb{R}_{\ge 0}^n$ is the set of points in $X_A$ with non-negative real coordinates. Similarly, we write $X_A^{>0} = X_A\cap \mathbb{R}_{>0}^n$ for the set of points in $X_A$ with positive real coordinates. If $A\in\mathbb{Z}_{\ge 0}^{k\times n}$ is such that $(1,\dotsc,1)\in\im A^T$, then the log-linear model $\mathcal{M}_A$ may be considered the projectivization of the non-negative toric variety $X_A^{\ge 0}$. Indeed, since $I_A$ is homogeneous, $X_A^{\ge 0}$ is a cone and we may scale each nonzero point so that the sum of the coordinates is equal to one. The algebraic moment map provides a way to understand the topology of the projectivization of the non-negative toric variety and hence, of our log-linear model $\mathcal{M}_A$.

The algebraic moment map is defined on the projectivization of a non-negative toric variety in \cite{Fulton,Sottile}. With the identification of the projectivization of the non-negative toric variety with the log-linear model $\mathcal{M}_A$, the algebraic moment map is defined as follows.

\begin{definition}
    The \textit{algebraic moment map} $\mu_A:\mathcal{M}_A\to \mathbb{R}^n$ is defined by $\mu_A(p) = Ap$. 
\end{definition}

For $p\in \mathcal{M}_A$, the image $\mu_A(p)$ is a convex combination of the columns of $A$. Hence, if $P_A$ denotes the polytope which is the convex hull of the columns of $A$, then the image of the algebraic moment map is contained in $P_A$. In fact, the algebraic moment map is a homeomorphism of $\mathcal{M}_A$ with $P_A$ as seen in \cite[Theorem 8.5]{Sottile}. Further, this map restricts to a homeomorphism of $\mathcal{M}_A^{>0}$ with the interior of the polytope $\interior(P_A)$ as described in \cite[Chapter 4.2]{Fulton}.

\begin{theorem}[\cite{Fulton,Sottile}]\label{thm:alg-moment-map}
If $A\in\mathbb{Z}_{\ge 0}^{k\times n}$ is such that $(1,\dotsc,1)\in\im A^T$, then the algebraic moment map $\mu_A:\mathcal{M}_A\to P_A$ is a homeomorphism. Further, the restriction $\mu_A:\mathcal{M}_A^{>0}\to\interior(P_A)$ is a homeomorphism. 
\end{theorem}

Thus, the topology of the log-linear model $\mathcal{M}_A$ is equivalent to that of the polytope $P_A$, and similarly for the topology of $\mathcal{M}_A^{>0}$ and $\interior(P_A)$. Hence, both $\mathcal{M}_A$ and $\mathcal{M}_A^{>0}$ are contractible spaces and in particular, they are connected. Further, the set $\mathcal{M}_A^{>0}$, consisting of points of the log-linear model with non-zero coordinates, is the interior of the log-linear model $\mathcal{M}_A$ and is dense in $\mathcal{M}_A$. We will make use of this connection in Section 4.

%%%%%%%%%%%%%%%%%%%%%%%%%%%%%%%%%%%%%%%%%%%%%%%%%%%%%%
%%%%COMPLETING TO THE TORIC VARIETY
\section{Completion to the Toric Variety}\label{sec:toricresults}
We first consider the completion problem to a toric variety. Fix an integer matrix $A\in\mathbb{Z}_{\ge 0}^{k\times n}$ and the corresponding toric variety $X_A$. Write $[n] = \{1,\dotsc,n\}$ for the set indexing the coordinates of $\mathbb{C}^n$. A subset $E\subseteq[n]$ determines a coordinate projection $\pi_E:\mathbb{C}^n\to\mathbb{C}^E$ to those coordinates indexed by $E$. We say a \textit{partial observation} is a point $p_E\in\mathbb{C}^E$, and a \textit{completion} of a partial observation $p_E$ to the toric variety $X_A$ is a point $p\in X_A$ that projects to $p_E$, $\pi_E(p) = p_E$. 

Our goal in this section is to determine the completable region $\pi_E(X_A)$, of partial observations which can be completed to the toric variety $X_A$. We accomplish this by first describing the image of the monomial map $\varphi^{A}$ and determining when a partial observation $p_E$ is completable to the image $\im\varphi^A$.

\subsection{The Image of the Monomial Map}
For a point $p = (p_1,\dotsc,p_n)\in\mathbb{C}^n$, the \textit{support} of $p$ is the set
\begin{align*}
\supp(p) = \{i\in[n]:p_i\ne 0\},
\end{align*}
consisting of the indices of non-zero coordinates of $p$. The following definition, taken from \cite{GeigerMeekSturmfels}, provides a necessary condition for a point to lie in the image $\im\varphi^A$.

\begin{definition}[\cite{GeigerMeekSturmfels}]\label{A-feasibility}
    Let $A\in\mathbb{Z}_{\ge 0}^{k\times n}$ be a matrix with column vectors $a_1,\dotsc,a_n\in\mathbb{Z}^k$. A point $p\in\mathbb{C}^n$ is \textit{$A$-feasible} if for every $j\in[n]\setminus\supp(p)$, the support $\supp(a_j)$ is not contained in the union $\bigcup_{l\in\supp(p)} \supp(a_l)$.
\end{definition}

We note that $A$-feasibility is equivalent to the notion of zero-consistency in \cite[Definition 2.1]{kahle2017geometry}. For points lying in the toric variety $p\in X_A$, $A$-feasibility is both necessary and sufficient for $p$ to lie in the image $\im\varphi^A$. This was proved in \cite[Theorem 3.1]{GeigerMeekSturmfels} for the non-negative toric variety, and we extend their result for the complex toric variety.

\begin{proposition}
    \label{image}
    If $A\in\mathbb{Z}_{\ge 0}^{k\times n}$ is an integer matrix, then the image of the map $\varphi^A$ is given by the set
    \begin{align*}
        \im\varphi^A = \{p\in X_A:p~\text{\normalfont is $A$-feasible}\}.
    \end{align*}
\end{proposition}
\begin{proof}
    Write $a_1,\dotsc,a_n$ for the columns of $A$ and $a_{ij}$ for the $(i,j)$-entry of $A$. If $p\in\im\varphi^A$, then $p\in\overline{\im\varphi^A} = X_A$. Similarly, if $p = \varphi^A(\theta)$ then for $j\in[n]\setminus\supp(p)$, there is an $i\in[k]$ such that $a_{ij} > 0$ and $\theta_i = 0$ so that $i\in\supp(a_j)$. However, $i\not\in\supp(a_l)$ for each $l\in\supp(p)$ since otherwise $p_l=0$. Therefore, $p$ is $A$-feasible. Thus, the inclusion $\im\varphi^A \subseteq \{p\in X_A:p~\text{\normalfont is $A$-feasible}\}$ holds.
    
    For the reverse inclusion, let $p\in X_A$ be $A$-feasible. Without loss of generality we may assume that $p$ only has non-zero entries. Indeed, since $p$ is $A$-feasible, we may restrict $\varphi^A$ to the coordinate subspace of $\mathbb{C}^k$ such that $\theta_i = 0$ for $i\in[n]\setminus\bigcup_{l\in\supp(p)}\supp(a_l)$. We compute a preimage of this restriction by discarding all zero entries of $p$, and reinserting them in the appropriate places afterwards.

    Let $p\in X_A$ have non-zero coordinates. From Proposition \ref{ToricIdeal}, for any $u\in\mathbb{Z}^n$ such that $Au = 0$, we must have that $p^u = 1$. By taking the logarithm and writing $\log(p)$ for the coordinate-wise logarithm of $p$, one finds that $u^T\log(p) = 0$. Thus, $\log(p)$ annihilates the kernel of $A$, or equivalently, lies in the image of $A^T$. Writing $\log(p) = A^Tv$ for some vector $v\in\mathbb{R}^k$ and applying coordinate-wise exponentiation, one finds that $p = \varphi^A(e^v)\in\im\varphi^A$. 
\end{proof}

If $A_E$ denotes the submatrix of $A$ whose columns are indexed by $E$, then there is an equality $\varphi^{A_E} = \pi_E\circ\varphi^A$. Thus, there is an equality of the images $\pi_E(\im\varphi^A) = \im\varphi^{A_E}$, and the defining ideal of the image $\pi_E(\im\varphi^A)$ is the toric ideal given by
\begin{align*}
I(\pi_E(\im\varphi^A)) = I(\im \varphi^{A_E}) = I(X_{A_E}) = I_{A_E}.
\end{align*}
This ideal may be computed via Proposition \ref{ToricIdeal}, or as an elimination ideal of $I_A$. Indeed, $I_{A_E}$ is obtained from $I_A$ by eliminating the variables $x_i$ for $i\in[n]\setminus E$. We obtain the following corollary of Proposition \ref{image} describing when a partial observation is completable to the image $\im\varphi^A$.

\begin{corollary}\label{cor:imagecompletion}
    Let $p_E\in\mathbb{C}^E$ be a partial observation.
    \begin{enumerate}
    \item[1.] $p_E$ is completable to a point of the image $\im\varphi^A$ if and only if $p_E\in X_{A_E}$ and $p_E$ is $A_E$-feasible. 
    
    \item[2.] If $p_E\in\mathbb{R}_{\ge 0}^E$ has non-negative real coordinates, then $p_E$ is completable to $\im\varphi^A$ if and only if there is a completion $p\in\im\varphi^A$ with non-negative real coordinates.
    \end{enumerate}
\end{corollary}
\begin{proof}
The first portion follows from the equality $\im\varphi^{A_E} = \pi_E(\im\varphi^A)$. A partial observation $p_E\in\mathbb{C}^E$ is completable to the image $\im\varphi^A$ if and only if it lies in the image $\im\varphi^{A_E}$. By Proposition \ref{image}, this is exactly when $p_E\in X_{A_E}$ and $p_E$ is $A_E$-feasible.

For the second portion, if $p_E\in\mathbb{R}_{\ge 0}^E$ is a partial observation which has non-negative real coordinates and is completable to a point $p\in\im\varphi^A$, then the point $|p|\in\im\varphi^A$ obtained by taking the coordinate-wise absolute value of $p$ is a completion of $p_E$ with non-negative real coordinates.
\end{proof}

\subsection{Completing to the Toric Variety} 
We now determine the set of partial observations which are completable to the toric variety $X_A$. We utilize the polyhedral structure of the polytope $P_A$, which is the convex hull of the columns of the matrix $A$. We note that as the column sums of $A$ are equal to $N>0$, the polytope $P_A$ is contained in the hyperplane determined by the equation $\sum_{i=1}^n p_i = N$ and has dimension $\rank A - 1$.

\begin{definition}
A \textit{facial set} of $A$ is a subset $F\subseteq[n]$ such that there is a vector $v\in\mathbb{R}^k$ satisfying $v^T a_i = 0$ for $i\in F$ and $v^T a_i>0$ for $i\in[n]\setminus F$. The vector $v$ is a \textit{inner normal vector} for the face $F$.
\end{definition}

Additionally, a \textit{facet} is a proper face which is not contained in any strictly larger proper face. Geometrically, a facial set indexes the columns of $A$ that lie in a face of the polytope $P_A$. Further, an inner normal vector $v$ of a facial set $F$ is an inner normal vector of a face of $P_A$. More on the relationship between polytopes and their faces can be found in \cite{Ewald}.

The following result from \cite[Lemma $A.2$]{GeigerMeekSturmfels} classifies the support of a point in $X_A$ via the facial sets of $A$.

\begin{lemma}[\cite{GeigerMeekSturmfels}]\label{lem:boundarypolytope}
    Let $A \in \mathbb{Z}_{\ge 0}^{k \times n}$ be a matrix whose columns are the vectors $a_1,\dotsc,a_n\in\mathbb{R}^k$. If $p\in X_A$, then $\supp(p)$ is a facial set of $A$.
\end{lemma}

If $F\subseteq[n]$ is a facial set of $A\in\mathbb{Z}_{\ge 0}^{k\times n}$, we define the \textit{characteristic vector} $\chi_F$ by $(\chi_F)_i = 1$ if $i\in F$ and $(\chi_F)_i = 0$ if $i\in[n]\setminus F$. By \cite[A.2]{GeigerMeekSturmfels}, $\chi_F\in X_A$. Thus, for every facial set $F$ of $A$, there is a point of $X_A$ with support $F$.

Lemma \ref{lem:boundarypolytope} and Corollary \ref{cor:imagecompletion} allow us to determine when a partial observation $p_E\in X_{A_E}$ is completable to the toric variety $X_A$. As the subset $E\subseteq[n]$ indexes the coordinates of $\mathbb{C}^E$, we may write $\supp(\pi_E(p))=\supp(p)\cap E$ for a point $p\in\mathbb{C}^n$.

\begin{theorem}
    \label{thm:toriccompletion}
    Let $p_E\in\mathbb{C}^E$ be a partial observation.
    \begin{enumerate}
    \item[1.] $p_E\in\mathbb{C}^E$ is completable to a point $p\in X_A$ if and only if $p_E\in X_{A_E}$ and there is a facial set $F$ of $A$ such that $\supp(p_E) = F\cap E$.

    \item[2.] If $p_E\in\mathbb{R}_{\ge 0}^E$ has non-negative real coordinates, then $p_E$ is completable to $X_A^{\ge 0}$ if and only if $p_E$ is completable to $X_A$.
    \end{enumerate}
\end{theorem}
\begin{proof}
If $p_E = \pi_E(p)\in \pi_E(X_A)$ for some $p\in X_A$, then by Lemma \ref{lem:boundarypolytope} $\supp(p)$ is a facial set of $A$ and $\supp(p_E)=\supp(p)\cap E$. 

Conversely, if $p_E\in X_{A_E}$ and $F$ is a facial set of $A$ such that $\supp(p_E)= E\cap F$, then $p_E$ is $A_{E\cap F}$-feasible since $p_E$ has non-zero coordinates $(p_E)_i$ for every $i\in E\cap F = \supp(p_E)$. By Corollary \ref{cor:imagecompletion}, this implies there exists $\theta\in\mathbb{C}^k$ such that the $i$-th coordinates $(\varphi^A(\theta))_i = (p_E)_i$ are equal for every $i\in E\cap F$. Consider the coordinate-wise product $\varphi^A(\theta)\chi_F\in X_A$. If $i\in E\cap F$, then the $i$-th coordinates $(\varphi^A(\theta)\chi_F)_i = (p_E)_i$ are equal. Similarly, if $i\in E\setminus F$, then $(\varphi^A(\theta)\chi_F)_i = (p_E)_i = 0$. Thus, $\varphi^A(\theta)\chi_F\in X_A$ is a completion of $p_E$ to $X_A$. 

For the second portion, if $p_E\in\mathbb{R}_{\ge 0}^E$ is a partial observation with completion $p\in X_A$ to $X_A$, then the coordinate-wise absolute value $|p|\in X_A^{\ge 0}$ is a completion to $X_A^{\ge 0}$.
\end{proof}

Given a partial observation $p_E\in\mathbb{C}^E$, Lemma \ref{lem:boundarypolytope} and Theorem \ref{thm:toriccompletion} allow us to quickly determine properties of completions to $X_A$ by studying the facial sets of $A$. If $F_1$ and $F_2$ are facial sets, then their intersection $F_1\cap F_2$ is a facial set. In particular, for any set $E\subseteq[n]$, there is a minimal facial set $F$ containing $E$, which is the intersection of all facial sets containing $E$. 

\begin{corollary}\label{cor:polytopefaces}
    Let $A\in\mathbb{Z}_{\ge 0}^{k\times n}$ be an integer matrix, $E\subseteq[n]$ be a subset of the coordinates, $p_E\in\mathbb{C}^E$ be a partial observation, and $p\in X_A$ be a completion of $p_E$. If $F\subseteq[n]$ is the smallest facial set containing $\supp(p_E)$, then $F\subseteq\supp(p)$. 
\end{corollary}

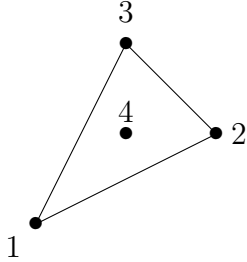
\begin{figure}[H]
    \centering
    \begin{tikzpicture}[scale=1.2]
    \draw (-1.25, -1.25) node {$1$};
    \draw (-1, -1) node {$\bullet$};
    \draw (1.25, 0) node {$2$};
    \draw (1, 0) node {$\bullet$};
    \draw (0, 1.35) node {$3$};
    \draw (0, 1) node {$\bullet$};
    \draw (0,0) node {$\bullet$};
    \draw (0, 0.25) node {$4$};

    \draw (-1, -1) -- (1,0);
    \draw (1, 0) -- (0, 1);
    \draw (0,1) -- (-1,-1);
    \end{tikzpicture}
    \caption{An example polytope $P_A$. The point corresponding to the 4-th column of $A$ lies in the relative interior of the polytope.}
    \label{fig:polytopeexample}
\end{figure}

\begin{example}\label{ex:polytope}
    Consider an integer matrix $A\in\mathbb{Z}_{>0}^{k\times 4}$ such that the polytope $P_A$ is as in Figure \ref{fig:polytopeexample}. If $E = \{4\}$ then the smallest face containing $E$ is the whole polytope $P_A$. Thus, the smallest facial set containing $E$ is $F = \{1,2,3,4\}$. By Corollary~\ref{cor:polytopefaces}, for any non-zero partial observation $p_E\in\mathbb{C}^E$, any completion to $X_A$ must have all non-zero coordinates. If $p_E = 0$, then any completion has support given by one of the proper facial sets, $\emptyset$, $\{1,2\}$, $\{1,3\}$, or $\{2,3\}$.
\end{example}

%%%%%%%%%%%%%%%%%%%%%%%%%%%%%%%%%%%%%%%%%%%%%%%%%%%%%%
%%%%COMPLETING TO THE LOG-LINEAR MODEL
\section{Completion to the Log-Linear Model}\label{sec:loglinearresults}
Let $E\subseteq[n]$ be a subset indexing some of the coordinates of $\mathbb{R}^n$ and $\pi_E:\mathbb{R}^n\to\mathbb{R}^E$ be the corresponding coordinate projection. Given an integer matrix $A\in\mathbb{Z}_{\ge 0}^{k\times n}$, we consider the completions of a partial observation $p_E\in\mathbb{R}^E$ to the log-linear model $\mathcal{M}_A$. Our analysis of the problem relies on our ability to understand completions to the non-negative toric variety $X_A^{\ge 0}$ as described in the previous section.

We provide a description of the interior and the boundary of the \textit{completable region} $\pi_E(\mathcal{M}_A)$, which consists of partial observations which can be completed to a point in $\mathcal{M}_A$: see Theorem \ref{thm:boundary-interior-image}. This is accomplished in Section \ref{sec:branch-locus} by analyzing the singular locus of the projection $\pi_E:\mathcal{M}_A^{>0}\to\mathbb{R}^E$, which is the locus of points where the differential $(d\pi_E)_p$ drops rank. For a point $p\in\mathcal{M}_A^{>0}$ that does not lie in the singular locus, the image $\pi_E(p)$ lies in the interior of the completable region $\pi_E(\mathcal{M}_A)$. Thus, the problem of determining the interior and the boundary of the completable region rests on understanding the image of the singular locus and the image of the boundary of the log-linear model $\partial\mathcal{M}_A = \mathcal{M}_A\setminus\mathcal{M}_A^{>0}$. 

In addition, we enumerate the completions for a partial observation with non-zero coordinates lying in the completable region. It will be shown that when there are finitely many completions, a partial observation can have either one or two completions to the log-linear model $\mathcal{M}_A$ depending on the subset $E\subseteq[n]$ chosen and whether the partial observation lies in the interior or the boundary of the completable region---see Theorem \ref{thm:boundary-completions} and Theorem \ref{thm:interior-completions}. Our results determine the number of completions of a partial observation to $\mathcal{M}_A$, but produce no general algorithm for computing these completions. 

We work under the mild assumption that our subset $E\subseteq[n]$ satisfies $|E| = \rank A_E = \rank A - 1 = \dim \mathcal{M}_A$. The assumption that $|E| = \rank A_E \le \dim\mathcal{M}_A$ guarantees that the image $\pi_E:\mathcal{M}_A\to\mathbb{R}^E$ is full-dimensional---see Corollary \ref{cor:dom-image}. Further, we assume that $|E| = \dim\mathcal{M}_A$ so that one expects only finitely many completions to the log-linear model $\mathcal{M}_A$ for a general partial observation. If $|E| < \dim\mathcal{M}_A$, then any partial observation lying in the completable region has infinitely many completions to $\mathcal{M}_A$ and we leave it as an open problem to describe the variety of completions in this case.

\subsection{The Singular Locus of a Coordinate Projection}\label{sec:branch-locus}
We note that for a map of varieties $\pi:X\to Y$, there is a maximal rank of the differential $d\pi_p:T_p X\to T_p Y$ for $p\in X$. Further, there is a Zariski open set (dense, open, path-connected set whose complement is a subvariety) $U\subseteq X$ for which this maximal rank is attained. The complement of this open set is the subvariety of $X$ where the rank of the differential drops.

\begin{definition}
The \textit{singular locus} of a map of varieties $\pi:X\to Y$ is the subvariety $X$ consisting of points $p\in X$ such that the differential $d\pi_p:T_p X\to T_p Y$ has rank less than the maximal rank.
\end{definition}

We begin by computing the tangent space at a point of $\mathcal{M}_A^{>0}$ considered as an open subset of a variety. Let $A\in\mathbb{Z}_{\ge 0}^{k\times n}$ have entries $a_{ij}$ for $1\le i\le k$ and $1\le j\le n$. By differentiating the monomial map $\varphi^A:\mathbb{R}^k_{>0}\to \mathbb{R}^n_{>0}$ at $\theta\in\mathbb{R}^k_{>0}$, we obtain the map $d\varphi_\theta^A:\mathbb{R}^k\to\mathbb{R}^n$ on tangent spaces defined by 
\begin{align*}
(d\varphi_\theta^A)_{ij} = a_{ji}\frac{1}{\theta_j}\varphi^A_i(\theta). 
\end{align*} 
By writing $p = \varphi^A(\theta)$, the image of the differential may be written as 
\begin{align*}
\im d\varphi^A_\theta = \{(p_1 v_1,\dotsc,p_n v_n)\in\mathbb{R}^n: v\in\im A^T\}.
\end{align*}
Since this image has dimension $\rank A = \dim X_A^{>0}$, it follows that the image $\im d\varphi_\theta^A$ coincides with the tangent space of the image $X_A^{>0}$ at $p$. Thus, $X_A^{>0}$ is smooth at each point and the tangent space is given by
\begin{align*}
T_p X_A^{>0} = \im d\varphi^A_\theta = \{(p_1 v_1,\dotsc,p_n v_n)\in\mathbb{R}^n: v\in\im A^T\}.
\end{align*}
Write $H\subseteq\mathbb{R}^n$ for the hyperplane defined by the equation $\sum_i x_i = 1$ so that $T_p H = \{x\in\mathbb{R}^n: \sum_i x_i = 0\}$. Since $(1,\dotsc,1)\in\im A^T$, it follows that $(p_1,\dotsc,p_n)\in T_p X_A^{>0}$ so that $T_p X_A^{>0} + T_p H = \mathbb{R}^n$---that is, $X_A^{>0}$ intersects $H$ transversally at each point. Thus, we may regard $\mathcal{M}^{>0}_A = X_A^{>0}\cap H$ as a smooth manifold of dimension $\dim \mathcal{M}_A^{>0} = \dim X_A^{>0} - 1 = \rank A - 1$. Further, we write its tangent space as
\begin{align*}
T_p \mathcal{M}_A^{>0} = \{(p_1 v_1,\dotsc,p_n v_n)\in\mathbb{R}^n: v\in\im A^T,~v^Tp = 0\}.
\end{align*}

We now fix a subset $E\subseteq[n]$ such that $|E| = \rank A_E = \rank A - 1$ and consider the corresponding coordinate projection $\pi_E:\mathbb{R}^n\to\mathbb{R}^E$ and its restriction to $\mathcal{M}^{>0}_A$. We show that its differential $(d\pi_E|_{\mathcal{M}_A^{>0}})_p$ is an isomorphism for most points $p\in\mathcal{M}_A^{>0}$, and give an explicit description of the set of points where this differential is not an isomorphism, or equivalently, of the singular locus of $\pi_E$. The following proposition will aid in identifying whether the differential $(d\pi_E|_{\mathcal{M}_A^{>0}})_p$ is an isomorphism. As a linear map, we identify the differential $(d\pi_E)_p$ with the map $\pi_E$ itself.

\begin{proposition}\label{prop:lin-alg}
Let $E\subseteq[n]$ is such that $|E| = \rank A_E = \rank A - 1$. Then $\dim(\im A^T \cap \ker\pi_E) = 1$. That is, every vector in $\im A^T \cap \ker\pi_E$ is a scalar multiple of any non-zero vector $\nu\in\im A^T \cap \ker\pi_E$.
\end{proposition}
\begin{proof}
The map $\pi_E:(\im A^T + \ker\pi_E)\to\im A_E^T$ gives an isomorphism of the space $(\im A^T + \ker\pi_E)/\ker\pi_E$ with $\im A_E^T$ so that 
\begin{align*}
\dim(\im A^T + \ker\pi_E) &= \dim \im A_E^T + \dim \ker\pi_E\\
&= \rank A_E + n - |E|\\
&= n.
\end{align*}
Thus, we may compute $\dim(\im A^T\cap \ker\pi_E) = \dim \im A^T + \dim \ker\pi_E - n = 1$.
\end{proof}

\begin{corollary}\label{cor:isom}
Let $E\subseteq[n]$ be such that $|E| = \rank A_E = \rank A - 1$ and $\nu\in \im A^T \cap \ker\pi_E$ be a non-zero vector. For $p\in\mathcal{M}_A^{>0}$, the differential $(d\pi_E|_{\mathcal{M}_A^{>0}})_p:T_p\mathcal{M}_A^{>0}\to\mathbb{R}^E$ is an isomorphism if and only if $\nu^Tp \ne 0$.
\end{corollary}
\begin{proof}
Since $\dim T_p\mathcal{M}_A^{>0}$ and $|E|$ are equal to $\rank A - 1$, the differential is an isomorphism if and only if it is injective. We use the description of the tangent space $T_p\mathcal{M}_A^{>0}$ found above,
\begin{align*}
T_p \mathcal{M}_A^{>0} = \{(p_1 v_1,\dotsc,p_n v_n)\in\mathbb{R}^n: v\in\im A^T,~v^Tp = 0\}.
\end{align*}
As $p\in\mathcal{M}_A^{>0}$, a non-zero tangent vector $(p_1 v_1,\dotsc, p_n v_n)\in T_p\mathcal{M}_A^{>0}$ lies in $\ker(d\pi_E|_{\mathcal{M}_A^{>0}})_p$ if and only if $v\in\im A^T \cap \ker \pi_E$ and $v^Tp = 0$. Equivalently, the kernel $\ker(d\pi_E|_{\mathcal{M}_A^{>0}})_p$ is trivial if and only if $\nu^Tp \neq 0$. 
\end{proof}

Corollary \ref{cor:isom} effectively describes the singular locus of $\pi_E$ as the set of points $p\in\mathcal{M}_A^{>0}$ such that $\nu^Tp = 0$ for a non-zero vector $\nu\in\im A^T\cap\ker \pi_E$. We demonstrate that this locus is a proper subset of $\mathcal{M}_A^{>0}$. Recall that the algebraic moment map $\mu_A:\mathcal{M}_A\to P_A$ is a homeomorphism as described in Theorem \ref{thm:alg-moment-map}.

\begin{proposition}\label{prop:branchlocuspolytope}
Let $E\subseteq[n]$ be such that $|E| = \rank A_E = \rank A - 1$ and let $\nu\in\im A^T \cap \ker\pi_E$ be a non-zero vector. A point $p\in\mathcal{M}_A$ satisfies $\nu^Tp = 0$ if and only if $\mu_A(p)\in\im A_E\cap P_A$.
\end{proposition}
\begin{proof}
Note that $q\in\im A_E$ if and only if $v^T q = 0$ for all $v\in\ker A_E^T$---that is, $q$ lies in the span of the columns of $A$ indexed by $E$ exactly when every hyperplane equation which vanishes on the columns of $A$ indexed by $E$ also vanish on $q$. Thus, for $p\in\mathcal{M}_A$, one has $\mu_A(p)\in\im A_E\cap P_A$ if and only if $v^T\mu_A(p) = (A^Tv)^Tp = 0$ for all $v\in\ker A_E^T$. However, the equality $A^T(\ker A_E^T) = \im A^T\cap \ker\pi_E$ holds so that the result follows. 
\end{proof}

The intersection $\im A_E\cap P_A$ is the set of points in $P_A$ spanned by the columns of $A$ indexed by $E$. Since $\rank A_E = \rank A - 1$, this is a proper subset of $P_A$ and the locus of points $p\in\mathcal{M}_A$ such that $\nu^Tp = 0$ is a proper subset of $\mathcal{M}_A$. Thus, the maximal rank of the differential $(d\pi_E|_{\mathcal{M}_A^{>0}})_p$ is $\dim \mathcal{M}_A^{>0} = |E| = \rank A - 1$ and this rank drops exactly on the locus of points in $\mathcal{M}_A^{>0}$ lying on the hyperplane defined by $\nu^Tp = 0$. Combining these results, we've proved the following.

\begin{corollary}\label{cor:branch-locus}
Let $E \subseteq [n]$ be such that $|E| = \rank A_E = \rank A - 1$ and $\nu\in\im A^T\cap \ker \pi_E$ any non-zero vector. The singular locus of the projection $\pi_E|_{\mathcal{M}_A}$ is the set of points
\begin{align*}
B_{A,E} = \{p\in\mathcal{M}_A^{>0}:\nu^Tp = 0\}. 
\end{align*}
\end{corollary}

We note that with this terminology, Proposition~\ref{prop:branchlocuspolytope} states that the image of the algebraic moment map applied to the singular locus $B_{A,E}$ is given by $\mu_A(B_{A,E}) = \im A_E\cap P_A$. That is, the image $\mu_A(B_{A,E})$ is the points of $P_A$ that are spanned by the columns of $A$ indexed by $E$. In addition, the equality $\im A^T\cap \ker\pi_E = A^T(\ker A_E^T)$ gives an effective method of computing a vector $\nu$. This is illustrated in Example \ref{ex:tri2pts}. 

\begin{example}\label{ex:tri2pts}
Let
\begin{align*}
A = \begin{pmatrix}
4 & 0 & 0 & 2 & 1 \\
0 & 4 & 0 & 1 & 2 \\
0 & 0 & 4 & 1 & 1
\end{pmatrix}
\in\mathbb{Z}_{\ge 0}^{3\times 5}
\end{align*}
and $E = \{4,5\}$. The polytope $P_A$ and the image $\mu_A(B_{A,E})$ are illustrated in Figure \ref{fig:tri2pts}. Since $|E| = \rank A_E = \rank A - 1 = 2$, the result of Proposition \ref{prop:branchlocuspolytope} applies. The kernel $\ker A_E^T$ is generated by the vector $\omega = (1,1,-3)$, so we may let $\nu = \frac{1}{4}A^T\omega = (1,1,-3,0,0)\in A^T(\ker A_E^T)$. Thus, the singular locus of the projection $\pi_E:\mathcal{M}_A^{>0}\to\mathbb{R}^E$ is given by the hyperplane section
\begin{align*}
B_{A,E} = \{p\in\mathcal{M}_A^{>0}: p_1+p_2-3p_3 = 0\}.
\end{align*}
\end{example}

\begin{figure}[H]
    \centering
    \begin{tikzpicture}[scale=1]
    \draw[blue,line width = .35mm] (0,1.2)--(2.8,1.2) node [left] at (0,1.2) {$\mu_A(B_{A,E})$};
    \draw (0,0) node {$\bullet$};
    \draw (4,0) node {$\bullet$};
    \draw (0,4) node {$\bullet$};
    \draw (1,1.2) node {$\bullet$};
    \draw (2,1.2) node {$\bullet$};
    \draw (-.2,-.2) node {$1$};
    \draw (4.2,-.2) node {$2$};
    \draw (-.2,4.2) node {$3$};
    \draw (.8,.9) node {$4$};
    \draw (2.2,.9) node {$5$};

    \draw[line width = .35mm] (0,0)--(4,0)--(0,4)--(0,0);
    \end{tikzpicture}
    \caption{The polytope $P_A$ and the image $\mu_A(B_{A,E})$ for Example \ref{ex:tri2pts}.}
    \label{fig:tri2pts}
\end{figure}
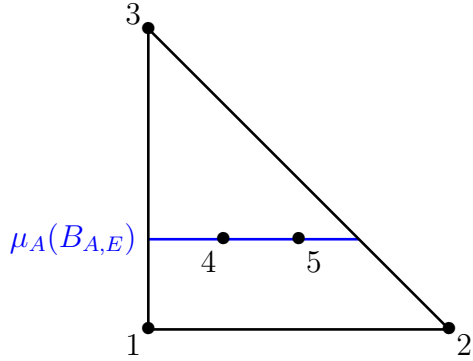

We end by showing that a point $p\in\mathcal{M}_A^{>0}\setminus B_{A,E}$ not in the singular locus maps to the interior of the completable region.

\begin{corollary}\label{cor:dom-image}
Let $E \subseteq [n]$ be such that $|E| = \rank A_E = \rank A - 1$ and let $p\in\mathcal{M}_A^{>0} \setminus B_{A,E}$. There is an open subset $U\subseteq\mathbb{R}^E$ such that $\pi_E(p)\in U$ and $U\subseteq\pi_E(\mathcal{M}_A)$. In particular, the completable image is full-dimensional and $\pi_E(p)\in\interior(\pi_E(\mathcal{M}_A))$ lies in the interior of the completable region.
\end{corollary}
\begin{proof}
For any point $p\in\mathcal{M}_A^{>0}\setminus B_{A,E}$, the differential $(d\pi_E|_{\mathcal{M}_A^{>0}})_p:T_p\mathcal{M}_A^{>0}\to\mathbb{R}^E$ is an isomorphism. By the implicit function theorem, $\pi_E$ restricts to a diffeomorphism of an open neighborhood of $p\in \mathcal{M}_A^{>0}$ to an open neighborhood of $\pi_E(p)$. 
\end{proof}

\subsection{The Interior and Boundary of the Completable Region}\label{sec:interior-completable-region}
Corollary \ref{cor:dom-image} shows that for $p\in\mathcal{M}_A^{>0}\setminus B_{A,E}$, the image $\pi_E(p)$ lies in the interior of the completable region $\pi_E(\mathcal{M}_A)$. We will show that the boundary of the completable region consists of the image of the boundary of the model $\pi_E(\partial\mathcal{M}_A)$ and the image of the singular locus $\pi_E(B_{A,E})$, which completely classifies the interior and the boundary of the completable region. We accomplish this by better understanding the completions of a partial observation to the non-negative toric variety $X_A^{\ge 0}$. 

\begin{proposition}\label{prop:exponential-trick}
Let $E\subseteq[n]$ be such that $|E| = \rank A_E = \rank A - 1$ and let $\nu\in\im A^T\cap \ker\pi_E$ be a non-zero vector. If $p_E\in\mathbb{R}_{>0}^E$ is a partial observation with non-zero coordinates, then for any completions $p,q\in X_A^{\ge 0}$ of $p_E$, there exists $\alpha\in\mathbb{R}$ such that 
\begin{align*}
p_i = q_i e^{\alpha \nu_i}~~\text{ for all }~i\in \supp(p)\cap\supp(q).
\end{align*}
\end{proposition}
\begin{proof}
We note that for any point $p\in X_A$, by projecting the coordinates indexed by $\supp(p)$, there is a completion to $\varphi^A(\mathbb{R}_{>0}^k)$ which agrees with $p$ on its support. Thus, we may write $p = \chi_{\supp(p)}\varphi^A(\theta)$ for some $\theta\in\mathbb{R}_{>0}^k$. Thus, we write $p = \chi_{F_1}\varphi^A(\theta_1)$ and $q = \chi_{F_2}\varphi^A(\theta_2)$ where $F_1=\supp(p)$ and $F_2=\supp(q)$ are facial sets that contain $E$, and $\theta_1,\theta_2\in\mathbb{R}_{>0}^k$. Then by taking the coordinate-wise logarithm, we have
\begin{align*}
0 = \log(p_E) - \log(q_E) = A_E^T(\log(\theta_1) - \log(\theta_2)). 
\end{align*}
Therefore $\log(\theta_1) - \log(\theta_2)\in\ker A_E^T$ so that $\log(p) - \log(q) = \alpha\nu \in A^T(\ker A_E^T)$ for some $\alpha\in\mathbb{R}$. Exponentiating, we have that for $\varphi^A(\theta_1) = \varphi^A(\theta_2)e^{\alpha \nu}$. Then for $i\in F_1\cap F_2$, we have that $p_i = q_i e^{\alpha\nu_i}$.
\end{proof}

We now leverage this proposition to show that the boundary of the completable region $\partial \pi_E(\mathcal{M}_A)$ is given by the image of the boundary of the model $\pi_E(\partial \mathcal{M}_A)$ and the image of the singular locus $\pi_E(B_{A,E})$. Equivalently, we show no point of the boundary $\partial\mathcal{M}_A$ or the singular locus $B_{A,E}$ map into the interior $\interior(\pi_E(\mathcal{M}_A))$. We start by showing a point $p\in\partial\mathcal{M}_A$ is such that either $\pi_E(p)$ has some coordinate equal to zero or for any $\epsilon > 0$, $(1+\epsilon)\pi_E(p)$ is not completable the log-linear model $\mathcal{M}_A$.

\begin{lemma}\label{lem:boundary-image}
Let $E\subseteq[n]$ be such that $|E| = \rank A_E = \rank A - 1$. The boundary of the model $\partial\mathcal{M}_A$ maps by $\pi_E$ into the boundary of the completable region $\partial\pi_E(\mathcal{M}_A)$. That is, there is an inclusion $\pi_E(\partial\mathcal{M}_A)\subseteq\partial\pi_E(\mathcal{M}_A)$. 
\end{lemma}
\begin{proof}
Let $p\in\mathcal{M}_A$ and $p_E = \pi_E(p)$. If $p\in\partial\mathcal{M}_A$, then some coordinate of $p$ is equal to zero and $\supp(p)\subseteq[n]$ is a proper facial set. There are two cases to consider. If $E$ is not contained in $\supp(p)$, then $p_E\in\mathbb{R}^E$ contains some zero coordinate and any open ball around $p_E$ contains points with negative coordinates. Those points with negative coordinates have no completion to $\mathcal{M}_A$ so that $p_E \in\partial\pi_E(\mathcal{M}_A)$. 

Assume that $E\subseteq \supp(p)$. Since $\rank A_E = \rank A - 1$ and $\supp(p)$ is a proper facial set, it follows that $\supp(p)$ is a facet and $\supp(p)$ is the smallest facial set containing $E$. Without loss of generality, we may let $\omega\in\ker A_E^T$ be an inner normal vector of the facet $\supp(p)$ so that $\nu = A^T\omega$ satisfies $\nu_i = 0$ for all $i\in\supp(p)$ and $\nu_i > 0$ for $i\in[n]\setminus\supp(p)$. 

Let $\epsilon>0$ and $q_E = (1+\epsilon)p_E\in\mathbb{R}^E$. Then $(1+\epsilon)p\in X_A^{\ge 0}$ is a completion of $q_E$ and by Proposition \ref{prop:exponential-trick}, any other completion $q\in X_A^{\ge 0}$ of $q_E$ satisfies
\begin{align*}
q_i = (1+\epsilon)p_i e^{\alpha \nu_i}
\end{align*}
for all $i\in\supp(p)\cap\supp(q)$. Since $E\subseteq\supp(q)$ and $\supp(p)$ is the smallest facial set containing $E$, we have that $\supp(p)\subseteq\supp(q)$. Further, since $\nu_i = 0$ for all $i\in \supp(p)$ and $q_i = (1+\epsilon)p_i$ for $i\in \supp(p)$. We compute

\begin{align*}
\sum_i q_i \ge \sum_{i\in\supp(p)} q_i = (1+\epsilon)\sum_{i\in\supp(p)} p_i = 1+\epsilon > 1.
\end{align*}
Any open set around $p_E$ then contains points which are not completable to $\mathcal{M}_A$ so that $p_E\in\partial\pi_E(\mathcal{M}_A)$. 
\end{proof}

Lemma \ref{lem:boundary-image} extends to the singular locus $B_{A,E}$ as well. We show that if $p\in B_{A,E}$, then the coordinates of $\pi_E(p)$ cannot be increased while remaining completable to the log-linear model $\mathcal{M}_A$. Thus, the image $\pi_E(B_{A,E})$ is contained in the boundary of the completable region as well.

\begin{lemma}\label{lem:branch-image}
Let $E\subseteq[n]$ be such that $|E| = \rank A_E = \rank A - 1$. The singular locus $B_{A,E}$ maps by $\pi_E$ into the boundary of the completable region $\partial\pi_E(\mathcal{M}_A)$. That is, there is an inclusion $\pi_E(B_{A,E})\subseteq\partial\pi_E(\mathcal{M}_A)$.  
\end{lemma}
\begin{proof}
We may assume $E$ is not contained in a proper facial set, since otherwise $B_{A,E}$ is empty as can be seen by Proposition \ref{prop:branchlocuspolytope}. Thus, by Corollary \ref{cor:polytopefaces}, if $p_E\in\mathbb{R}^E$ has non-zero coordinates, any completion $p\in X_A^{\ge 0}$ has non-zero coordinates. Let $p\in B_{A,E}$ and $p_E = \pi_E(p)$. If $\epsilon>0$ and $q_E = (1+\epsilon)p_E$ as above, then $(1+\epsilon)p\in X_A^{\ge 0}$ is a completion. For any other completion $q\in X_A^{\ge 0}$, there exists $\alpha\in\mathbb{R}$ such that $q_i = (1+\epsilon)p_i e^{\alpha \nu_i}$ for all $i\in[n]$ since $p$ and $q$ have all non-zero coordinates. Since $\nu^Tp = 0$, one computes
\begin{align*}
\sum_i q_i = (1+\epsilon)\sum_i p_i e^{\alpha\nu_i} \ge (1+\epsilon)\sum_i p_i(1+\alpha\nu_i) = 1+\epsilon>1.
\end{align*}
Again one finds that any open set around $p_E$ contains points which are not completable to $\mathcal{M}_A$ so that $p_E\in\partial \pi_E(\mathcal{M}_A)$. 
\end{proof}

We combine Lemma \ref{lem:boundary-image} and Lemma \ref{lem:branch-image} to describe the boundary and the interior of the completable region.

\begin{theorem}\label{thm:boundary-interior-image}
Let $E\subseteq[n]$ be such that $|E| = \rank A_E = \rank A - 1$. The boundary of the completable region is equal to the image by $\pi_E$ of the union of the boundary of the model $\partial\mathcal{M}_A$ and the singular locus $B_{A,E}$. Precisely,
\begin{align*}
\partial \pi_E(\mathcal{M}_A) = \pi_E(\partial\mathcal{M}_A) \cup \pi_E(B_{A,E}). 
\end{align*}
Equivalently, the interior of the completable region is equal to the image
\begin{align*}
\interior(\pi_E(\mathcal{M}_A)) = \pi_E(\mathcal{M}_A^{>0}\setminus B_{A,E}). 
\end{align*}
\end{theorem}
\begin{proof}
For $p\in\mathcal{M}_A^{>0}\setminus B_{A,E}$, the differential $(d\pi_E)_p$ is an isomorphism so that $\pi(p)\in\interior(\pi(\mathcal{M}_A))$. By taking the relative complement with $\pi_E(\mathcal{M}_A)$, the inclusion $\partial\pi_E(\mathcal{M}_A)\subseteq \pi_E(\partial\mathcal{M}_A) \cup \pi_E(B_{A,E})$ follows. The reverse inclusion is an immediate consequence of Lemma \ref{lem:boundary-image} and Lemma \ref{lem:branch-image}.
\end{proof}

We illustrate Theorem \ref{thm:boundary-interior-image} in Example \ref{ex:square-part1}.

\begin{example}\label{ex:square-part1}
Let
\begin{align*}
A = \begin{pmatrix}
2 & 1 & 1 & 0 \\
0 & 1 & 0 & 1 \\
0 & 0 & 1 & 1
\end{pmatrix}\in\mathbb{Z}_{\ge 0}^{3\times 4}
\end{align*}
and $E = \{1,4\}$. The polytope $P_A$, the image $\mu_A(B_{A,E})$, and the completable region $\pi_E(\mathcal{M}_A)$ are illustrated in Figure \ref{fig:not-blob}. Using coordinates $x,y,z,w$ for $\mathbb{R}^4$, the log-linear model is given by
\begin{align*}
\mathcal{M}_A = \{(x,y,z,w)\in\mathbb{R}_{\ge 0}^4: xw - yz = x+y+z+w-1 = 0\}. 
\end{align*}
As $(0,1,-1)\in\ker A_E^T$, we may take $\nu = (0,1,-1,0)$ so that 
\begin{align*}
B_{A,E} = \{(x,y,z,w)\in\mathcal{M}_A^{>0}: y - z = 0\}.
\end{align*}
As the subset $E$ is not contained in a proper facial set, the boundary $\partial\mathcal{M}_A$, of points where some coordinate is equal to zero, maps to the coordinate axes of $\mathbb{R}_{\ge 0}^E$. Contrary to this, the singular locus $B_{A,E}$ maps to the curved boundary of Figure \ref{fig:not-blob}. The defining equation for this this curve can be obtained via elimination. Indeed, the defining ideal of the singular locus is the prime ideal
\begin{align*}
I(B_{A,E}) = \langle xw-yz,x+y+z+w-1,y - z\rangle.
\end{align*}
Using \texttt{Macaulay2} \cite{M2} to eliminate the variables $y$ and $z$, we find that the eliminant of $I(B_{A,E})$ is generated by the single polynomial
\begin{align*}
f(x,w) = x^2 - 2xw - w^2 - 2x - 2w + 1.
\end{align*}
Thus, $f(x,w) = 0$ is the defining equation of the Zariski closure of the projection $\overline{\pi_E(B_{A,E})}$.

We note that the image $\pi_E(B_{A,E})$ is a semialgebraic set and the zero set $\{(x,w)\in\mathbb{R}^2:f(x,w) = 0\}$ is its Zariski closure. In general, it is difficult to obtain a semialgebraic description of the image $\pi_E(B_{A,E})$ or of the completable region $\pi_E(\mathcal{M}_A)$. 
\end{example}

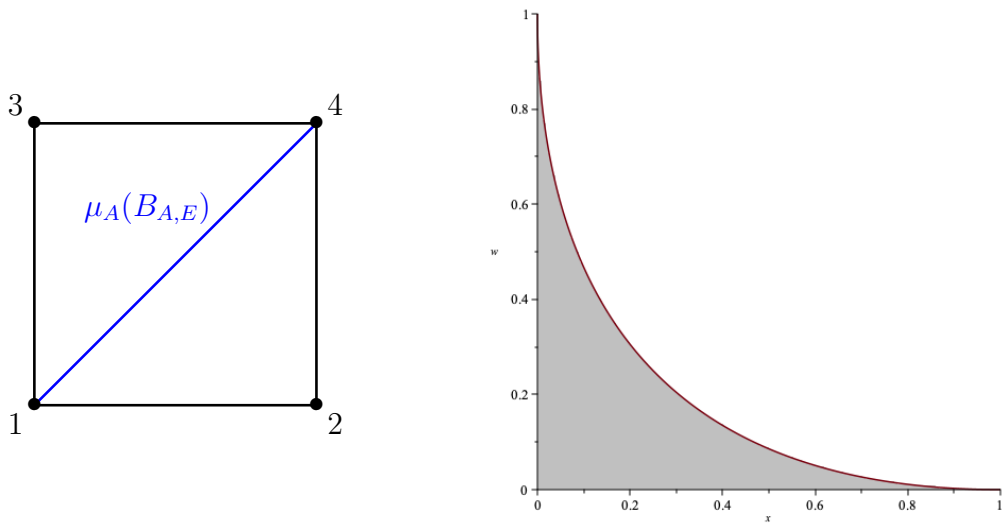
\begin{figure}[H]
    \centering
        \begin{minipage}{.45\textwidth}
            \centering
            \begin{tikzpicture}[scale=1.25]
            \draw[blue,line width = .35mm] (0,0)--(3,3) node [above left] at (2,1.8) {$\mu_A(B_{A,E})$};
            \draw (0,0) node {$\bullet$};
            \draw (3,0) node {$\bullet$};
            \draw (0,3) node {$\bullet$};
            \draw (3,3) node {$\bullet$};
            \draw (-.2,-.2) node {$1$};
            \draw (3.2,-.2) node {$2$};
            \draw (-.2,3.2) node {$3$};
            \draw (3.2,3.2) node {$4$};

            \draw[line width = .35mm] (0,0)--(3,0)--(3,3)--(0,3)--(0,0);
            \end{tikzpicture}
            
        \end{minipage}
        \begin{minipage}{.45\textwidth}
            \centering
            \includegraphics[scale = 0.4]{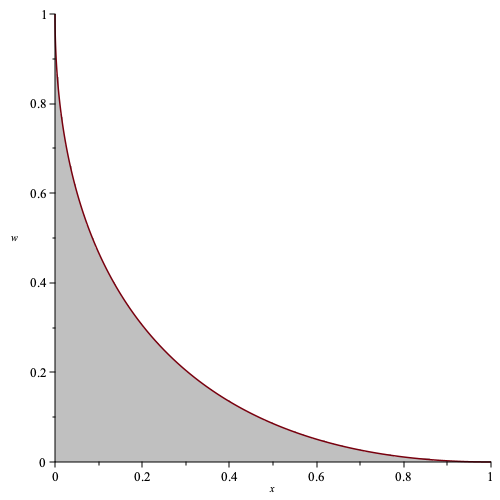}
        \end{minipage}
        \centering
        \caption{(Left) The polytope $P_A$ and image $\mu_A(B_{A,E})$ for Example \ref{ex:square-part1} \\ (Right) The completable region $\pi_E(\mathcal{M}_A)$ for Example \ref{ex:square-part1}}
    \label{fig:not-blob}
\end{figure}

\subsection{Enumerating Completions to the Log-Linear Model}\label{sec:enumerating-completions}
Given a matrix $A\in\mathbb{Z}_{\ge 0}^{k\times n}$ and a subset $E\subseteq[n]$ such that $|E| = \rank A_E = \rank A - 1$, we discuss the possible number of completions of a partial observation with non-zero coordinates in the completable region $p_E\in\pi_E(\mathcal{M}_A)$ to the log-linear model $\mathcal{M}_A$. By elementary methods, we conclude that such a partial observation has either 0, 1, or 2 completions to $\mathcal{M}_A$ depending on whether the subset $E\subseteq[n]$ lies in a proper facial set and whether $p_E$ lies in the boundary or the interior of the completable region.

We first consider the case that a partial observation with non-zero coordinates lies in the boundary of the completable region $p_E\in\partial\pi_E(\mathcal{M}_A)$. 

\begin{theorem}\label{thm:boundary-completions}
Let $E\subseteq[n]$ be such that $|E| = \rank A_E = \rank A - 1$. If $p_E\in\partial\pi_E(\mathcal{M}_A)$ is a partial observation with non-zero coordinates lying in the boundary of the completable region, then $p_E$ has a unique completion to the log-linear model $\mathcal{M}_A$.
\begin{enumerate}
    \item[1.] If $E$ is contained in a proper facial set, then this completion lies in the boundary of the model $\partial\mathcal{M}_A$. 

    \item[2.] If $E$ is not contained in a proper facial set, then this completion lies in the singular locus $B_{A,E}$.
\end{enumerate}
\end{theorem}
\begin{proof}
Let $p_E\in\partial\pi_E(\mathcal{M}_A)$ be a partial observation with non-zero coordinates in the boundary of the completable region. By Theorem \ref{thm:boundary-interior-image}, we may express the boundary of the completable region as
\begin{align*}
\partial\pi_E(\mathcal{M}_A) = \pi_E(\partial\mathcal{M}_A)\cup\pi_E(B_{A,E}).
\end{align*}
If $E$ is contained in a proper facial set, then $B_{A,E}$ is empty by Proposition~\ref{prop:branchlocuspolytope} so that $p_E$ has a completion $p\in\partial\mathcal{M}_A$. Similarly, if $E$ is not contained in a proper facial set, then any completion of $p_E$ has all non-zero coordinates so that $p_E$ necessarily has a completion $p\in B_{A,E}$. We separate the proof into these cases:

\vspace{\baselineskip}

If $E$ is contained in a proper facial set, then $p_E$ has a completion $p\in\partial\mathcal{M}_A$. Similar to Lemma \ref{lem:boundary-image}, since $\rank A_E = \rank A - 1$ and $\supp(p)$ is a proper facial set, it follows that $\supp(p)$ is a facet and the smallest face containing $E$. We may let $\omega\in\ker A_E^T$ be an inner normal vector of the facial set $\supp(p)$ so that $\nu = A^T\omega$ satisfies $\nu_i =  0$ for $i\in\supp(p)$ and $\nu_i>0$ for $i\in[n]\setminus\supp(p)$. 

If $q\in\mathcal{M}_A$ is any other completion of $p_E$, then by Proposition \ref{prop:exponential-trick} there exists $\alpha\in\mathbb{R}$ such that
\begin{align*}
q_i = p_i e^{\alpha\nu_i}
\end{align*}
for $i\in\supp(p)\cap\supp(q)$. Since $E\subseteq\supp(q)$ and $\supp(p)$ is the smallest facial set containing $E$, we have $\supp(p)\subseteq\supp(q)$. Further, since $\nu_i = 0$ for $i\in\supp(p)$, we have that $q_i = p_i$ for $i\in\supp(p)$. Since $p,q\in\mathcal{M}_A$, we see
\begin{align*}
0=\sum_{i \in [n]} (q_i - p_i) = \sum_{i\in[n]\setminus\supp(p)} q_i.
\end{align*}
Since $q_i\ge 0$ for each $i$, this implies that $q_i = 0$ for $i\in[n]\setminus\supp(p)$ so that $q = p$. 

\vspace{\baselineskip}
If $E$ is not contained in a proper facial set, then $p_E$ has a completion $p\in B_{A,E}$. If $q\in\mathcal{M}_A$ is any other completion, then $q$ also has non-zero coordinates. Let $\nu\in\im A^T\cap \ker\pi_E$ be a non-zero vector. By Proposition \ref{prop:exponential-trick}, there exists $\alpha\in\mathbb{R}$ such that
\begin{align*}
q_i = p_i e^{\alpha\nu_i}
\end{align*}
for all $i\in[n]$. Recall that since $p \in B_{A,E}$ we have $\nu^Tp = 0$ by Corollary \ref{cor:branch-locus}. Observe that
\begin{align*}
\sum_{i \in [n]} p_i(e^{\alpha\nu_i} - \alpha \nu_i - 1) = \sum_{i \in [n]} (q_i - p_i) - \alpha\nu^Tp = 0.
\end{align*}
As each term of this summand is non-negative, this implies each summand is equal to zero. Since $p$ has non-zero coordinates and $\nu$ is a non-zero vector, there is some index $i$ such that $p_i\ne 0$ and $\nu_i\ne 0$, but $p_i(e^{\alpha\nu_i} - \alpha\nu_i - 1) = 0$. Therefore it must be that $\alpha = 0$ and $q = p$.

\vspace{\baselineskip}

\end{proof}

We turn our attention now to partial observations lying in the interior of the completable region $p_E\in\interior(\pi_E(\mathcal{M}_A))$. Completions in this case may not be unique, however it can be seen that there are at most two.

\begin{proposition}\label{prop:inequalities}
Let $E\subseteq[n]$ be such that $|E| = \rank A_E = \rank A - 1$ and let $\nu\in\im A^T\cap \ker\pi_E$ be a non-zero vector. If $p_E\in\interior(\pi_E(\mathcal{M}_A))$ is a partial observation in the interior of the completable region, then any two distinct completions $p,q\in\mathcal{M}_A$ are such that $\nu^Tp$ and $\nu^Tq$ are non-zero and have opposite signs. 
\end{proposition}
\begin{proof}
By Theorem \ref{thm:boundary-interior-image}, $p_E$ has a completion $p\in\mathcal{M}_A^{>0}\setminus B_{A,E}$ and any other completion also lies in $\mathcal{M}_A^{>0}\setminus B_{A,E}$. By Proposition \ref{prop:exponential-trick}, if $p,q\in\mathcal{M}_A^{>0}\setminus B_{A,E}$ are any two completions, there exists $\alpha\in\mathbb{R}$ such that
\begin{align*}
q_i = p_i e^{\alpha \nu_i}
\end{align*}
for all $i\in [n]$. Without loss of generality by interchanging $p$ and $q$, we may assume that $\alpha>0$. Observe
\begin{align*}
0 = \sum_i (q_i - p_i) = \sum_i p_i(e^{\alpha \nu_i} - 1) \ge \alpha \sum_i \nu_i p_i = \alpha \nu^Tp.
\end{align*}
Since $\alpha>0$ and $\nu^Tp\ne 0$, this implies that $\nu^Tp<0$. By the same process reversing the roles of $p$ and $q$ and by replacing $\alpha$ by $-\alpha$, one finds that $\nu^Tq >0$. 
\end{proof}

\begin{corollary}\label{cor:two-completions}
Let $E\subseteq[n]$ be such that $|E| = \rank A_E = \rank A - 1$, and let $p_E\in\interior(\pi_E(\mathcal{M}_A))$ be a partial observation in the interior of the completable region. Then $p_E$ has at most two completions to $\mathcal{M}_A$. 
\end{corollary}

The number of completions to $\mathcal{M}_A$ for a partial observation in the interior of the completable region may vary. Indeed, the number of completions is depends on whether the set $E\subseteq[n]$ is contained in a proper facial set.

\begin{theorem}\label{thm:interior-completions}
Let $E\subseteq[n]$ be such that $|E| = \rank A_E = \rank A - 1$ and let $p_E\in\interior(\pi_E(\mathcal{M}_A))$ be a partial observation lying in the interior of the completable region.
\begin{enumerate}
    \item[1.] If $E$ is contained in a proper facial set, then $p_E$ has a unique completion to $\mathcal{M}_A$.

    \item[2.] If $E$ is not contained in a proper facial set, then $p_E$ has two completions to $\mathcal{M}_A$.
\end{enumerate}
\end{theorem}
\begin{proof}
\noindent\textit{1.} If $E\subseteq[n]$ is contained in a proper facial set $F$, we may let $\omega\in \ker A_E^T$ be an inner normal vector for $F$. Then $\nu = A^T\omega$ is such that $\nu_i = 0$ for $i\in F$ and $\nu_i>0$ for $i\in[n]\setminus F$. Thus, $\nu^Tp\ge 0$ for all $p\in\mathcal{M}_A$. From Proposition \ref{prop:inequalities}, it follows that $p_E$ must have a unique completion to $\mathcal{M}_A$. 

\vspace{\baselineskip}

\noindent\textit{2.} Assume $E\subseteq[n]$ is not contained in a proper facial set. If $p_E\in\interior(\pi_E(\mathcal{M}_A))$ is a partial observation with non-zero coordinates in the interior of the completable region, then by Theorem \ref{thm:boundary-interior-image}, $p_E$ has a completion $p\in\mathcal{M}_A^{>0}\setminus B_{A,E}$. By Corollary \ref{cor:polytopefaces} and Proposition \ref{prop:exponential-trick}, any completion $q\in X_A^{\ge 0}$ must have all non-zero coordinates and there must exist $\alpha\in\mathbb{R}$ satisfying
\begin{align*}
q_i = p_i e^{\alpha\nu_i}
\end{align*}
for all $i\in[n]$. Setting $x = e^\alpha$, the completion $q\in X_A^{\ge 0}$ lies in the log-linear model $\mathcal{M}_A$ exactly when $x$ is a root of the (Laurent) polynomial
\begin{align*}
f(x) = \sum_i p_i x^{\nu_i} - 1.
\end{align*}
Conversely, any positive root of this polynomial yields a completion of $p_E$ lying in the log-linear model $\mathcal{M}_A$. From Corollary \ref{cor:two-completions}, the polynomial $f$ has at most two positive roots. We use continuity arguments to show that $f$ always has two positive roots. 

Note that since $E$ is not contained in a proper facial set, the vector $\nu$ must have positive and negative coordinates corresponding to terms of $f$ with positive and negative exponents. As these terms of $f$ have positive coefficients, $f(x)$ can then be made arbitrarily large for $x>1$ sufficiently large and for $0<x<1$ sufficiently small. Note that $x=1$ is a root of $f$ and $f'(1) = \nu^Tp$. Thus, if $\nu^Tp>0$, then for $0<x<1$ sufficiently close to $x=1$, $f(x)$ is negative. Thus $f(x)$ has a root in the open interval $(0,1)$. Similarly, if $\nu^Tp<0$, then for $x>1$ sufficiently close to $x=1$, $f(x)$ is negative and $f(x)$ has a root in the open interval $(1,\infty)$. Thus, $f$ always has two positive roots yielding two completions of $p_E$. 
\end{proof}

Example \ref{ex:square-part2} demonstrates how the proof of Theorem \ref{thm:interior-completions} may be used to recover all completions of a partial observation $p_E\in\interior(\pi_E(\mathcal{M}_A))$ from a single completion.

\begin{example}\label{ex:square-part2}
As in Example \ref{ex:square-part1}, let
\begin{align*}
A = \begin{pmatrix}
2 & 1 & 1 & 0 \\
0 & 1 & 0 & 1 \\
0 & 0 & 1 & 1
\end{pmatrix}\in\mathbb{Z}_{\ge 0}^{3\times 4}
\end{align*}
and $E = \{1,4\}$. Recall that the log-linear model is given by
\begin{align*}
\mathcal{M}_A = \{(x,y,z,w)\in\mathbb{R}_{\ge 0}^4: xw - yz = x+y+z+w-1 = 0\}, 
\end{align*}
and the singular locus is the hyperplane section of $\mathcal{M}_A^{>0}$ defined by
\begin{align*}
B_{A,E} = \{(x,y,z,w)\in\mathcal{M}_A^{>0}: y - z = 0\}.
\end{align*}

Consider the partial observation $(\frac{1}{6},\frac{1}{3})\in\mathbb{R}_{\ge 0}^E$ and a completion $p = (\frac{1}{6},\frac{1}{3},\frac{1}{6},\frac{1}{3})\in\mathcal{M}_A^{>0}\setminus B_{A,E}$. Since $\nu^Tp = \frac{1}{6}>0$, there is another completion $q\in\mathbb{M}_A^{>0}$ such that $\nu^Tq <0$. This completion can be obtained by computing positive roots of the Laurent polynomial 
\begin{align*}
f(x) = \sum_i p_i x^{\nu_i} - 1 = \frac{1}{3}x-\frac{1}{2}+\frac{1}{6}x^{-1}.
\end{align*}
One finds that $x=\frac{1}{2}$ is a root of $f(x)$ so that 
\begin{align*}
q = p\left(\frac{1}{2}\right)^{(0,1,-1,0)} = \left(\frac{1}{6},\frac{1}{6},\frac{1}{3},\frac{1}{3}\right)\in\mathcal{M}_A^{>0}\setminus B_{A,E}
\end{align*}
is the other completion. Theorem \ref{thm:interior-completions} implies that these are the only completions of the partial observation $p_E$ to the log-linear model $\mathcal{M}_A$. 
\end{example}

%%%%%%%%%%%%%%%%%%%%%%%%%%%%%%%%%%%%%%%%%%%%%%%%%%%%%%
%%%%DESCRIBING THE COMPLETABLE REGION
\section{Describing the Completable Region}\label{sec:computing-completable-region}
Given a subset $E\subseteq[n]$ and an integer matrix $A\in\mathbb{Z}_{\ge 0}^{k\times n}$, the completable region $\pi_E(\mathcal{M}_A)$ is a full-dimensional semialgebraic set---it is defined by polynomial inequalities. A \textit{semialgebraic description} of the completable region $\pi_E(\mathcal{M}_A)$ is such a description of the completable region by polynomial inequalities. Obtaining a semialgebraic description of the completable region is difficult in general. However, in this section we will provide examples where obtaining the complete semialgebraic description of the completable region is possible by elementary methods. 
From Theorem \ref{thm:boundary-interior-image}, the algebraic boundary of the completable region $\pi_E(\mathcal{M}_A)$ is the union of the Zariski closure of the image of the boundary $\pi_E(\partial\mathcal{M}_A)$ and the Zariski closure of the image of the singular locus $\pi_E(B_{A,E})$. In both cases, the defining ideal can be computed explicitly via elimination as long as the defining ideal of the boundary $\partial\mathcal{M}_A$ and the defining ideal of the singular locus $B_{A,E}$ are known. 
But first we consider the simpler problem of computing defining equations for the boundary of the completable region as in Example \ref{ex:square-part1}. 

\begin{definition}
The \textit{algebraic boundary} of the completable region $\pi_E(\mathcal{M}_A)$ is the Zariski closure of the Euclidean boundary $\partial\pi_E(\mathcal{M}_A)$.
\end{definition}

The defining ideal of the boundary $\partial\mathcal{M}_A$ is generated by the product of the coordinates $\prod_i x_i$, the hyperplane $\sum_i x_i - 1$, and $I_A$. However, the defining ideal of $B_{A,E}$ my be difficult to obtain. For instance, given $\nu\in\im A^T\cap \ker\pi_E$, the hyperplane section defined by $\nu^Tp = 0$ may intersect the Zariski closure of the log-linear model $\mathcal{M}_A$ in more than one component. However, from Proposition \ref{prop:branchlocuspolytope}, at most one of these components intersects the model $\mathcal{M}_A$. Techniques from real algebraic geometry such as the Positivestellensatz and its variants produce methods of choosing this component in general. See \cite{roy2020bound,fatemeh2022synthesis} for more information on these techniques and their use in applications. 

We provide pseudo-code for computing the defining ideal of the algebraic boundary. We note the reliance on several methods which we consider as black-boxes: one for computing the toric ideal $I_A$, one for producing the minimal primes of an ideal, one for determining whether an ideal has a positive zero, and one for computing elimination ideals.

\begin{algorithm}[Defining Ideal of the Algebraic Boundary]\

\def\arraystretch{1.5}
\noindent\begin{tabular}{ll}
\hline
\textbf{Input} & $A$: An integer matrix in $\mathbb{Z}_{\ge 0}^{k\times n}$\\
~ & $E$: A subset such that $|E| = \rank A_E = \rank A - 1$\\ 
\hline
\textbf{Output} & $I$: The defining ideal of the algebraic boundary.\\
\hline
\end{tabular}

\begin{enumerate}
\item[1.] Find a non-zero vector $\nu\in A^T(\ker A_E^T)$.
\item[2.] Compute the minimal primes of the ideal $J = I_A + \langle \sum_i x_i - 1, \nu^Tx\rangle$.
\item[3.] Compute $I(B_{A,E})$ as the minimal prime of $J$ whose corresponding variety intersects the log-linear model $\mathcal{M}_A$.
\item[4.] Set $I_1$ as the ideal obtained by eliminating the variables $x_i$ for $i\in[n]\setminus E$ from the defining ideal of the singular locus $I(B_{A,E})$.
\item[5.] Compute $I(\partial\mathcal{M}_A)$ as the ideal $I_A + \langle \sum_i x_i - 1,\prod_i x_i\rangle$.
\item[6.] Set $I_2$ as the ideal obtained by eliminating the variables $x_i$ for $i\in[n]\setminus E$ from the defining ideal of the boundary of the model $I(\partial\mathcal{M}_A)$.
\item[7.] Return $I$ as the radical of the ideal $I_1 I_2$.
\end{enumerate}
\label{alg:algebraic-boundary}
\end{algorithm}

We note that in general, inequalities obtained from the defining equations of the algebraic boundary do not give a semialgebraic description of the completable region. This is demonstrated in Example \ref{ex:blob}.

\begin{example} 
\label{ex:blob}
Let

\begin{equation*}
    A = \begin{pmatrix}
    3 & 4 & 0 & 3 & 2 \\ 
    0 & 1 & 2 & 1 & 1 \\
    2 & 0 & 3 & 1 & 2    
    \end{pmatrix}\in\mathbb{Z}_{\ge 0}^{3\times 5}
\end{equation*}
and $E = \{4,5\}$. We use coordinates $x$, $y$, $z$, $u$, $v$ for $\mathbb{R}_{\ge 0}^5$. The polytope $P_A$ is a triangle containing two interior points as illustrated in Figure \ref{fig:blob} and the defining ideal $I_A$ is generated by three binomials,
\begin{align*}
    I_A = \langle u^2 - yv, v^3 - xzu, uv^2 - xyz \rangle.  
\end{align*}

We first consider the portion of the algebraic boundary which is the image of the singular locus. The kernel $\ker A_E^T$ is generated by the single vector $\omega = (1,-4,1)$ so that we may write $\nu = (5,0,-5,0,0)$ and 
\begin{align*}
B_{A,E} = \{(x,y,z,u,v)\in\mathcal{M}_A^{>0}:x-z = 0\}.
\end{align*}

\begin{figure}[H]
    \centering
    \begin{minipage}{.45\textwidth}
    \centering
    \begin{tikzpicture}
    \draw (-0.25, 2) node[left] {$1$};
    \draw (0, 2) node {$\bullet$};
    \draw (1, -0.25) node[below] {$2$};
    \draw (1, 0) node {$\bullet$};
    \draw (2.25, 3.25) node[above] {$3$};
    \draw (2, 3) node {$\bullet$};
    \draw (1,1) node {$\bullet$};
    \draw (1, 0.9) node[below] {$4$};
    \draw (1,2) node {$\bullet$};
    \draw (1, 1.9) node[below] {$5$};

    \draw (0, 2) -- (1,0);
    \draw (0, 2) -- (2, 3);
    \draw (1,0) -- (2,3);
    \end{tikzpicture}
    \end{minipage}
    \begin{minipage}{.45\textwidth}
    \centering
    \includegraphics[scale = 0.4]{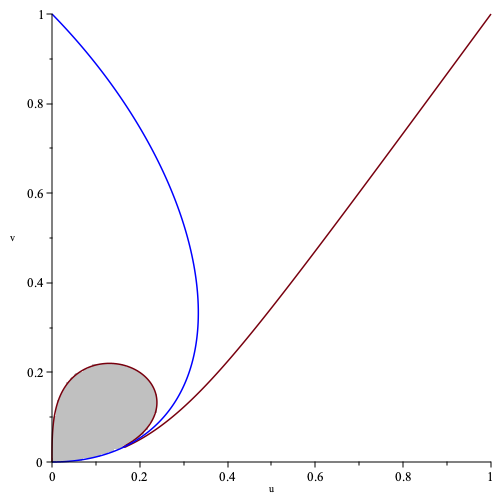}
    \end{minipage}
    \caption{(Left) The polytope $P_A$ for Example \ref{ex:blob}\\(Right) The completable region and curves defined by {\color{black!60!red}$f(u,v) = 0$} and {\color{black!30!blue}$g(u,v) = 0$}}
    \label{fig:blob}
\end{figure}
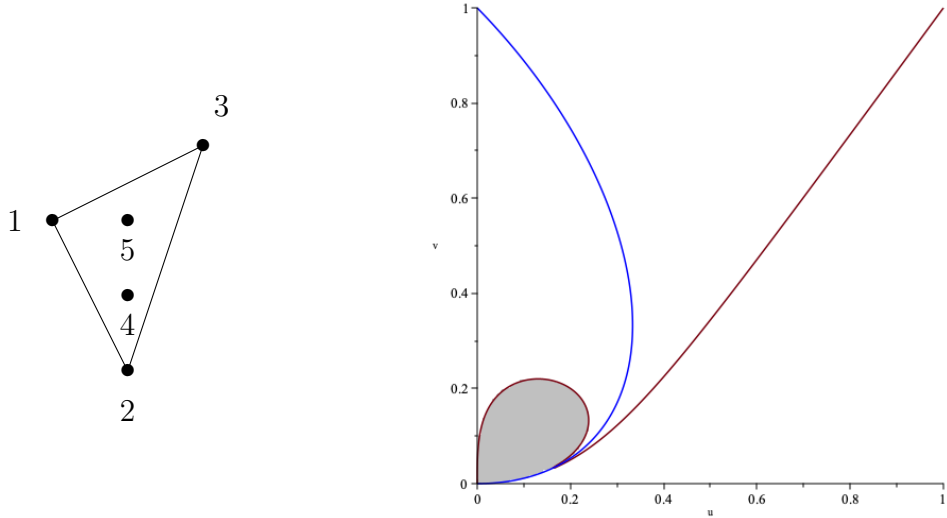

As the ideal 
\begin{align*}
I(B_{A,E}) = \langle u^2 - yv, v^3 - xzu, uv^2 - xyz,x+y+z+u+v-1,x-z\rangle,
\end{align*}
is prime, it is the defining ideal of the singular locus. Using \texttt{Macaulay2} to eliminate the variables $x$, $y$, and $z$ from the ideal $I(B_{A,E})$, the eliminant is generated by the single polynomial
\begin{align*}
f(u,v) = u^5+2u^4v+3u^3v^2+2u^2v^3+uv^4-4v^5-2u^3v-2u^2v^2-2uv^3+uv^2.
\end{align*}
Thus, the algebraic boundary of the completable region $\pi_E(\mathcal{M}_A)$ is defined by the single polynomial equation $f(u,v) = 0$. 

Note if $p\in\partial\mathcal{M}_A$, then some coordinate is equal to zero so that $\supp(p)$ is a proper facial set. Since the only facial set intersecting $E$ is the set $[5] = \{1,2,3,4,5\}$, it follows that $\supp(p)\cap E = \emptyset$ and $\pi_E(p) = (0,0)$. That is, $\pi_E(\partial\mathcal{M}_A) = \{(0,0)\}$. Thus, the algebraic boundary is defined by the single equation $f(u,v) = 0$. 

Contrary to this, the completable region is not described by either inequality $f(u,v)\ge 0$ or $f(u,v)\le 0$. Indeed, the completable region is illustrated in Figure \ref{fig:blob}. The red curve is the algebraic boundary defined by the equation $f(u,v) = 0$, but only the shaded area is the completable region. 

To obtain a semialgebraic description of the completable region, consider a partial observation $(u,v)\in\pi_E(\mathcal{M}_A)$ with non-zero coordinates and a completion $(x,y,z,u,v)\in\mathcal{M}_A$. Multiplying through the equation $x+y+z+u+v = 1$ by $zuv$ and using the generators of the toric ideal $I_A$, one finds that $z$ satisfies the quadratic equation
\begin{align*}
uv z^2 + (u^3 + u^2v + uv^2 - uv)z + v^4 = 0.
\end{align*}
Conversely, using the generators of $I_A$, any solution to this quadratic yields a completion of $(u,v)$ to the log-linear model $\mathcal{M}_A$. Thus, this quadratic has all non-negative solutions and the sum of these solutions 
\begin{align*}
\frac{-1}{uv}\left(u^3+u^2v+uv^2-uv\right) \ge 0
\end{align*}
must also be non-negative. Since $u$ and $v$ are non-zero, this implies that
\begin{align*}
g(u,v) = -u^2-uv-v^2+v \ge 0.
\end{align*}
The blue curve in Figure \ref{fig:blob} is defined by the equation $g(u,v) = 0$ and together the inequalities $f(u,v) \ge 0$ and $g(u,v)\ge 0$ provide a semialgebraic description of the completable region. Precisely,
\begin{align*}
\pi_E(\mathcal{M}_A) = \{(u,v)\in\mathbb{R}_{\ge 0}^2: f(u,v)\ge 0,~g(u,v)\ge 0\}. 
\end{align*}

\end{example}

As in Example \ref{ex:blob}, it is often the case that the ideal $I_A + \langle \sum_i x_i - 1, \nu^Tx\rangle$ is prime, in which case it is equal to the defining ideal of the singular locus $I(B_{A,E})$. The following example demonstrates a situation in which the boundary of the log-linear model $\mathcal{M}_A$ maps onto the boundary of the completable region.

\begin{example}
A hierarchical model is a log-linear model determined by the data of a simplicial complex with positive integer weights given at each of the vertices. Each vertex corresponds to a random variable with the weight of the vertex being the number of outcome states of the variable. The faces of the simplicial structure encode dependencies among these variables. We consider the hierarchical model associated to the length two segment with each vertex weight equal to two. This corresponds to three binary random variables $X_1$, $X_2$, and $X_3$ with the probability of their outcomes related by certain relations determined by the dependencies.

\begin{figure}[H]
    \centering
    \begin{tikzpicture}
    \draw (0, -0.35) node {$X_1$};
    \draw (0, 0) node {$\bullet$};
    \draw (2, -0.35) node {$X_2$};
    \draw (2, 0) node {$\bullet$};
    \draw (4, -0.35) node {$X_3$};
    \draw (4, 0) node {$\bullet$};

    \draw (0, 0) -- (2,0);
    \draw (2, 0) -- (4, 0);
    \end{tikzpicture}
    \caption{The length two line segment}
    \label{fig:LengthTwoSegment}
\end{figure}
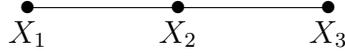

The model has coordinates $p_{ijk}$ given by the probability that $X_1 = i$, $X_2 = j$, and $X_3 = k$. As described in Chapter 9 of \cite{sullivant2018algebraic}, a hierarchical model is a log-linear model and so can be described by an integer matrix $A\in\mathbb{Z}_{\ge 0}^{k\times n}$. For the hierarchical model associated to Figure \ref{fig:LengthTwoSegment} and all weights equal to two, the matrix $A$ is given by
\begin{align*}
A = &\begin{pmatrix}
1 & 1 & 0 & 0 & 0 & 0 & 0 & 0 \\
0 & 0 & 1 & 1 & 0 & 0 & 0 & 0 \\
0 & 0 & 0 & 0 & 1 & 1 & 0 & 0 \\
0 & 0 & 0 & 0 & 0 & 0 & 1 & 1 \\
1 & 0 & 0 & 0 & 1 & 0 & 0 & 0 \\
0 & 1 & 0 & 0 & 0 & 1 & 0 & 0 \\
0 & 0 & 1 & 0 & 0 & 0 & 1 & 0 \\
0 & 0 & 0 & 1 & 0 & 0 & 0 & 1
\end{pmatrix},
\end{align*}
where the columns of $A$ index the coordinates $p_{ijk}$ in lexicographical order. One finds that $\rank A = 6$ and the toric ideal $I_A$ is generated by the following binomials
\begin{align*}
I_A = \langle p_{112}p_{211} - p_{111}p_{212},p_{122}p_{211} - p_{121}p_{222}\rangle.
\end{align*}

We demonstrate two subsets $E\subseteq[n]$ for which a semialgebraic description of the completable region $\pi_E(\mathcal{M}_A)$ can be easily computed and show how a semialgebraic description can be obtained in general for this model. 

Let $E = \{1,2,3,4,5\}\subseteq[8]$ so that $|E| = \rank A_E = \rank A - 1 = 5$. In other words, we will observe the coordinates $p_{111}$, $p_{112}$, $p_{121}$, $p_{122}$, and $p_{211}$. \texttt{Macaulay2} can be used to verify that $E$ is contained in the proper facial set $\{1,2,3,4,5,6\}$. Thus, the singular locus $B_{A,E}$ is empty and the algebraic boundary can be obtained by eliminating $p_{212}$, $p_{221}$, and $p_{222}$ from the ideal $I_A + \langle \sum_{i,j,k} p_{ijk} - 1, \prod_{i,j,k} p_{ijk}\rangle$. We find that the algebraic boundary is given by the single polynomial
\begin{align*}
f(p_{111},p_{112},p_{121},p_{122},p_{211}) &= -p_{111}^2-p_{111}p_{121}-p_{111}p_{122}-p_{111}p_{211} - p_{112}p_{211} + p_{111}.
\end{align*}

Further, a semialgebraic description of the completable region is given by
\begin{align*}
\pi_E(\mathcal{M}_A) = \{p_E\in\mathbb{R}_{\ge 0}^5:f(p_{111},p_{112},p_{121},p_{122},p_{211})\ge 0\}.
\end{align*}
Indeed, given a partial observation $p_E\in\pi_E(\mathcal{M}_A)$ with non-zero coordinates, the remaining coordinates can be computed in rational functions of the observed coordinates $p_{111}$, $p_{112}$, $p_{121}$, $p_{122}$, and $p_{211}$ as
\begin{align*}
p_{212} &= \frac{p_{112}p_{211}}{p_{111}}\\
p_{221} &= \frac{p_{121}f(p_{111},p_{112},p_{121},p_{122},p_{211})}{p_{111}(p_{121}+p_{122})}\\
p_{222} &= \frac{p_{122}f(p_{111},p_{112},p_{121},p_{122},p_{211})}{p_{111}(p_{121}+p_{122})}.
\end{align*}
As our partial observation has non-zero coordinates, this completion lies in the log-linear model $\mathcal{M}_A$ exactly when $f(p_{111},p_{112},p_{121},p_{122},p_{211})\ge 0$. 

We note that for any subset $E\subseteq[n]$ which is contained in a proper facial set, the above analysis can be carried out to obtain a semialgebraic description of the completable region. Indeed, for the hierarchical model $\mathcal{M}_A$, if $E\subseteq[n]$ is contained in a proper facial set, then the coordinates $p_i$ for $i\in[n]\setminus E$ can be explicitly computed in rational functions in the coordinates $p_i$ for $i\in E$ via the generators of the toric ideal $I_A$ and the relation $\sum_i p_i = 1$. 

For a subset $E\subseteq[n]$ such that $|E| = \rank A_E = \rank A - 1$ that is not contained in a proper facial set, then by choosing a coordinate $p_{ijk}$ not indexed by $E$, the remaining coordinates are rational functions in $p_{ijk}$ and the coordinates indexed by $E$. Indeed, by observing the generators of the toric ideal $I_A$, each of these rational functions has the form $f(p_E)p_{ijk}$ or $f(p_E)p_{ijk}^{-1}$ where $f(p_E)$ is a (Laurent) monomial in the coordinates indexed by $E$. Subsituting these values into the relation $\sum_{ijk} p_{ijk} = 1$, one finds that $p_{ijk}$ satisfies a quadratic relation. As in Example \ref{ex:blob}, this quadratic has all non-negative roots for $p_E\in\pi_E(\mathcal{M}_A)$ and one obtains an additional polynomial inequality that the partial observation $p_E$ must satisfy. Since the remaining coordinates are given as rational functions with positive coefficients in $p_{ijk}$ and the coordinates indexed by $E$, any solution of this quadratic corresponds to a completion in $\mathcal{M}_A$ of $p_E$. Thus, one obtains a semialgebraic description of the completable region.

In either of the cases above, given a subset $E\subseteq[n]$ satisfying $|E| = \rank A_E = \rank A - 1$, one can obtain a semialgebraic description of the completable region $\pi_E(\mathcal{M}_A)$ for this hierarchical model. 
\end{example}

\bibliographystyle{alpha}
\bibliography{arxiv_submission}

\end{document}